\documentclass[reqno, 10pt]{amsart}
\usepackage{geometry,amsmath,amssymb,bbm}
\newcommand{\assign}{:=}
\newcommand{\longhookrightarrow}{{\lhook\joinrel\relbar\joinrel\rightarrow}}
\newcommand{\mathd}{\mathrm{d}}
\newcommand{\nocomma}{}
\newcommand{\nosymbol}{}
\newcommand{\tmop}[1]{ {#1} }
\newcommand{\tmtextit}[1]{{\itshape{#1}}}
\newtheorem{lemma}{Lemma}

\newtheorem{remark}{Remark} 
\newtheorem{theorem}{Theorem}

\begin{document}

\title[Compressible NS Equations]{On some Liouville Type Theorems for the Compressible Navier-Stokes
Equations}

\author[D. Li]{Dong Li}
\address[D. Li]
{Department of Mathematics, University of British Columbia, 1984 Mathematics Road,
Vancouver, BC, Canada V6T1Z2}
\email{dli@math.ubc.ca}

\author[X. Yu]{Xinwei Yu}
\address[X. Yu]
{Department of Mathematical and Statistical Sciences, University of Alberta, 632 CAB, Edmonton, AB, Canada T6G2G1}
\email{xinweiyu@math.ualberta.ca}



\begin{abstract}
  We prove several Liouville type results for stationary
  solutions of the $d$-dimensional compressible Navier-Stokes equations. In
  particular, we show that when the dimension $d \geqslant 4$, the natural
  requirements $\rho \in L^{\infty} \left( \mathbbm{R}^d \right)$, $v \in
  \dot{H}^1 \left( \mathbbm{R}^d \right)$ suffice to guarantee that the
  solution is trivial. For dimensions $d=2,3$, we assume the extra
  condition $v \in L^{\frac{3d}{d-1}}(\mathbb R^d)$. This improves a
  recent result of Chae \cite{Chae2012}.
\end{abstract}

\maketitle

\section{Introduction}

In this paper we consider the stationary barotropic compressible
Navier-Stokes equations on $\mathbbm{R}^d$, $d \geqslant 2$
\begin{align}
  \nabla \cdot \left( \rho v \right) & =\; 0,  \label{eq:NS-density}\\
  \nabla \cdot \left( \rho v \otimes v \right) & =\;  - \nabla P + \mu
  \triangle v + \left( \mu + \lambda \right) \nabla \left( \nabla \cdot v
  \right),  \label{eq:NS-momentum}\\
  P \left( \rho \right) & =\;  A \rho^{\gamma}, \hspace{2em} \gamma > 1.
  \label{eq:NS-pressure}
\end{align}
Here $\rho, v, P$ are the density, velocity and pressure of the
fluid respectively. The coefficients $\lambda, \mu, A$ satisfy $\mu
> 0, \lambda + 2 \mu > 0 \nocomma$, $A > 0$. This system is a
reduction of the compressible Navier-Stokes equations by assuming
that the volumetric entropy $s$ remains constant. For more
discussions on the compressible Navier-Stokes equations and its
various related models, see e.g.
\cite{Feireisl2004,Lions1998,Novotny2004}.

For the system (\ref{eq:NS-density}--\ref{eq:NS-pressure}), it turns out that,
if we put suitable global integrability conditions on $\rho$ and/or $v$,
then the only possible solution is the trivial one $v = 0, \rho =$constant.
For example, if we assume $\rho \in L^{\infty} \cap C^{\infty}, v \in
C_c^{\infty}$, we can then multiply the momentum equation
(\ref{eq:NS-momentum}) by $v$ and integrate over space:
\begin{equation}
  \int_{\mathbbm{R}^d} \left[ \nabla \cdot \left( \rho v \otimes v \right)
  \right] \cdot v \mathd x = - \int_{\mathbbm{R}^d} \nabla P \cdot v \mathd x
  + \mu \int_{\mathbbm{R}^d} \left( \triangle v \right) \cdot v \mathd x +
  \left( \mu + \lambda \right)  \int_{\mathbbm{R}^d} \left[ \nabla \left(
  \nabla \cdot v \right) \right] \cdot v \mathd x.
\end{equation}
Integration by parts gives
\begin{equation}
  \int_{\mathbbm{R}^d} \left[ \nabla \cdot \left( \rho v \otimes v \right)
  \right] \cdot v \mathd x = \int_{\mathbbm{R}^d}
  \Bigl[(\rho v \cdot \nabla) v \Bigr] \cdot
  v \mathd x = \int_{\mathbbm{R}^d} \rho v \cdot \nabla \left( \frac{v^2}{2}
  \right) \mathd x = - \int_{\mathbbm{R}^d} \nabla \cdot \left( \rho v \right)
  \frac{v^2}{2} \mathd x = 0,
\end{equation}
and
\begin{equation}
  \int_{\mathbbm{R}^d} \nabla P \cdot v \mathd x = \int_{\mathbbm{R}^d}
  \left( A \rho^{\gamma} \right) \cdot v \mathd x = A \int_{\mathbbm{R}^d}
  \rho v \cdot \nabla \rho^{\gamma - 1} \mathd x = - A \int_{\mathbbm{R}^d}
  \nabla \cdot \left( \rho v \right) \rho^{\gamma - 1} \mathd x = 0.
\end{equation}
These lead to
\begin{equation}
  \mu \int_{\mathbbm{R}^d} \left| \nabla v \right|^2 \mathd x + \left( \mu +
  \lambda \right)  \int_{\mathbbm{R}^d} \left( \nabla \cdot v \right)^2 \mathd
  x = 0 \Longrightarrow v \equiv \text{\tmop{constant}.}
\end{equation}
As $v \in C_c^{\infty}$ we obtain $v\equiv 0$. Since
$\nabla v = 0$ we also get from (\ref{eq:NS-momentum}) that $\nabla
\left( A \rho^{\gamma} \right) = \nabla P = 0$. Hence $\rho \equiv$constant too.

Thus we have easily shown that, if $\rho \in L^{\infty} \cap C^{\infty}, v
\in C_c^{\infty}$, then the only solution to
(\ref{eq:NS-density}--\ref{eq:NS-pressure}) is the trivial one. However such a simple
argument does not work any more once we relax the compact support condition on $v$
and weaken the regularity assumptions. Very
recently, Chae \cite{Chae2012} has shown that the above conclusion, $v = 0,
\rho =$constant, still holds as long as
\begin{equation}
  \left\| \rho \right\|_{L^{\infty} \left( \mathbbm{R}^d \right)} + \left\|
  \nabla v \right\|_{L^2 \left( \mathbbm{R}^d \right)} + \left\| v
  \right\|_{L^{\frac{d}{d - 1}} \left( \mathbbm{R}^d \right)} < \infty
  \hspace{2em} \text{\tmop{when} } d \leqslant 6, \label{eq:Chae-d=6}
\end{equation}
\begin{equation}
  \left\| \rho \right\|_{L^{\infty} \left( \mathbbm{R}^d \right)} + \left\|
  \nabla v \right\|_{L^2 \left( \mathbbm{R}^d \right)} + \left\| v
  \right\|_{L^{\frac{d}{d - 1}} \left( \mathbbm{R}^d \right)} + \left\| v
  \right\|_{L^{\frac{3 d}{d - 1}} \left( \mathbbm{R}^d \right)} < \infty
  \hspace{2em} \text{\tmop{when} } d \geqslant 7. \label{eq:Chae-d=7}
\end{equation}
The condition
\begin{align*}
\| \rho\|_{L^{\infty}(\mathbb R^d)} + \| \nabla v \|_{L^2(\mathbb R^d)} <\infty
\end{align*}
is very natural as most physical flows have bounded density and finite enstrophy. On
the other hand, the integrability condition
$
v \in L^{\frac d {d-1}} (\mathbb R^d)
$
is a fairly strong assumption since physical flows with finite energy need not
belong to this class.
In this paper we will remove this assumption and weaken further the integrability conditions on the
velocity field $v$. Our main
results are the following.

\begin{theorem}
  \label{thm:main}Let the dimension $d \geqslant 2$. Suppose $\left( \rho, v
  \right)$ is a smooth solution to (\ref{eq:NS-density}--\ref{eq:NS-pressure})
  satisfying
  \begin{align}
    \left\| \rho \right\|_{L^{\infty} \left( \mathbbm{R}^d \right)} < \infty,
    \hspace{2em} v \in \dot{H}^1 \left( \mathbbm{R}^d \right), &  &
    \text{\tmop{if} } d \geqslant 4;  \label{eq:cond-d=4}\\
    \left\| \rho \right\|_{L^{\infty} \left( \mathbbm{R}^d \right)} + \left\|
    \nabla v \right\|_{L^2 \left( \mathbbm{R}^d \right)} + \left\| v
    \right\|_{L^{\frac 92} \left( \mathbbm{R}^d \right)} < \infty, &  &
    \text{\tmop{if} } d = 3;  \label{eq:cond-d=3}\\
    \left\| \rho \right\|_{L^{\infty} \left( \mathbbm{R}^d \right)} + \left\|
    \nabla v \right\|_{L^2 \left( \mathbbm{R}^d \right)} + \left\| v
    \right\|_{L^6 \left( \mathbbm{R}^d \right)} < \infty, &  & \text{\tmop{if} }
    d = 2;  \label{eq:cond-d=2}
  \end{align}
  then $v = 0$ and $\rho =$constant.
\end{theorem}

\begin{remark}
  The space $\dot{H}^1 \left( \mathbbm{R}^d \right)$ in (\ref{eq:cond-d=4}) is
  defined in the usual way as the completion of $C_c^{\infty} \left(
  \mathbbm{R}^d \right)$ under the norm $\left\| \nabla f \right\|_{L^2 \left(
  \mathbbm{R}^d \right)}$. In essence, (\ref{eq:cond-d=4}) is the condition
  $\left\| \nabla v \right\|_{L^2 \left( \mathbbm{R}^d \right)} < \infty$
  together with the requirement that $v$ decays to $0$ at infinity. Such decay
  is necessary to conclude $v = 0$ as only derivatives of $v$ appear in
  (\ref{eq:NS-density}--\ref{eq:NS-pressure}), while $v$ itself doesn't. Also
  note that conditions (\ref{eq:cond-d=4}--\ref{eq:cond-d=2}) are weaker than
  (\ref{eq:Chae-d=6}--\ref{eq:Chae-d=7}) in the sense that $v$ can decay
  more slowly at infinity.
\end{remark}

\begin{remark}
We stress that in dimensions $d=2,3$, the extra condition $v \in L^{\frac {3d}{d-1}}$ is
very mild. For example, if $v\in H^1(\mathbb R^d)$, then by Sobolev embedding we have
$v \in L^{\frac {3d}{d-1}}$ for $d=2,3$. Thus Theorem \ref{thm:main} asserts that there are no
nontrivial stationary solutions with the velocity field in the natural energy class.
\end{remark}
If we are willing to put more integrability condition on $v$, then the condition on
$\nabla v$ can even be dropped.

\begin{theorem}
  \label{thm:all-d}Let $d \geqslant 2$. Suppose $\left( \rho, v \right)$ is a
  smooth solution to (\ref{eq:NS-density}--\ref{eq:NS-pressure}) satisfying
  \begin{equation}
    \left\| \rho \right\|_{L^{\infty} \left( \mathbbm{R}^d \right)} + \left\|
    v \right\|_{L^{ \frac d  {d - 1}} \left( \mathbbm{R}^d \right)} +
    \left\| v \right\|_{L^{ \frac{3 d} {d - 1}} \left( \mathbbm{R}^d
    \right)} < \infty, \label{eq:cond-all-d}
  \end{equation}
  then $v = 0$ and $\rho =$constant on $\mathbbm{R}^d$.
\end{theorem}

\begin{remark}
It is not difficult to check that
the condition \eqref{eq:cond-all-d} is weaker than \eqref{eq:Chae-d=6}--\eqref{eq:Chae-d=7}
for all $d\ge 2$.

\end{remark}
Finally, it is also possible to put the integrability condition on $\nabla \cdot
v$ instead of $\nabla v$.

\begin{theorem}
  \label{thm:div-v}Let $d \geqslant 2$. Suppose $\left( \rho, v \right)$ is a
  smooth solution to (\ref{eq:NS-density}--\ref{eq:NS-pressure}) satisfying
  \begin{equation}
    \left\| \rho \right\|_{L^{\infty} \left( \mathbbm{R}^d \right)} + \left\|
    \nabla \cdot v \right\|_{L^2 \left( \mathbbm{R}^d \right)} + \left\| v
    \right\|_{L^{\frac{3 d} { d - 1}} \left( \mathbbm{R}^d \right)} <
    \infty, \hspace{2em} \text{\tmop{if} } d \le 3; \label{eq:cond-d=234}
  \end{equation}
  \begin{equation}
    \left\| \rho \right\|_{L^{\infty} \left( \mathbbm{R}^d \right)} + \left\|
    \nabla \cdot v \right\|_{L^2 \left( \mathbbm{R}^d \right)} + \left\| v
    \right\|_{L^{ \frac{3 d}  {d - 1}} \left( \mathbbm{R}^d \right)} +
    \left\| v \right\|_{L^{ \frac{2 d}  {d - 2}} \left( \mathbbm{R}^d
    \right)} < \infty, \hspace{2em} \text{\tmop{if} } d \ge 4;
    \label{eq:cond=d=5}
  \end{equation}
  then $v = 0$ and $\rho =$constant on $\mathbbm{R}^d$.
\end{theorem}

\begin{remark}
One can also check that
the condition \eqref{eq:cond-d=234}--\eqref{eq:cond=d=5} is
weaker than \eqref{eq:Chae-d=6}--\eqref{eq:Chae-d=7}
for all $d\ge 2$. It is an interesting question whether
the assumption on $\|\nabla \cdot v\|_2$ can be dropped.
\end{remark}

The following sections are devoted to the proofs of Theorems \ref{thm:main},
\ref{thm:all-d}, \ref{thm:div-v}.

We conclude the introduction by setting up some notations.

\subsubsection*{Notations}
\begin{itemize}
\item For any two quantities $X$ and $Y$, we denote $X \lesssim Y$ if
$X \le C Y$ for some harmless constant $C>0$. Similarly $X \gtrsim Y$ if $X \ge CY$ for some
$C>0$. We denote $X \sim Y$ if $X\lesssim Y$ and $Y \lesssim X$.
We shall write $X\lesssim_{Z_1,Z_2,\cdots, Z_k} Y$ if $X \le CY$ and we want to stress that the constant
$C$ depends on the quantities $(Z_1,\cdots, Z_k)$. Similarly we define $\gtrsim_{Z_1,\cdots, Z_k}$.

\item We shall denote by $\|f\|_p=\| f\|_{L^p}=\| f\|_{L^p(\mathbb R^d)}$ the usual Lebesgue
norm of a scalar or vector valued function $f$ on $\mathbb R^d$. For any Lebesgue measurable set
$A \subset \mathbb R^d$, we denote by $|A|$ the Lebesgue measure of $A$.

\item For any two scalar functions $f$, $g$ on $\mathbb R^d$, we use the notation $f*g$ to denote
the standard convolution
\begin{align*}
(f*g)(x)= \int_{\mathbb R^d} f(x-y) g(y)dy,
\end{align*}
whenever the above integral is well-defined.

\end{itemize}

\section{Integrability Lemmas}

In the assumptions of Theorems \ref{thm:main}, \ref{thm:all-d}, \ref{thm:div-v} we
only require $\rho \in L^{\infty} \left( \mathbbm{R}^d \right)$ which gives $P
\in L^{\infty} \left( \mathbbm{R}^d \right)$ through the constitutive relation
(\ref{eq:NS-pressure}). However, to rigorously carry out the integration by parts
argument, we need more integrability on the pressure $P$ and the density $\rho$.
A natural idea is to make use of  the momentum equation (\ref{eq:NS-momentum})
and upgrade the integrability of $(P,\rho)$ through the integrability of the velocity
field $v$. To this end, we re-write (\ref{eq:NS-momentum}) as
\begin{equation}
  \nabla P = - \nabla \cdot \left( \rho v \otimes v \right) + \mu \triangle v
  + \left( \mu + \lambda \right) \nabla \left( \nabla \cdot v \right) .
\end{equation}
After removing the gradient on both sides of the above equation, it
 clear that integrability of $v$ may imply some integrability of $P$ (modulo some
 constants),
which in turn could lead to integrability of $\rho$ through
(\ref{eq:NS-pressure}). We present the precise relations in the following two
lemmas. They will play a key role in the proofs of our main theorems.

\begin{lemma}
  \label{lem:p2}Let $P \in L^{\infty} \left( \mathbbm{R}^d \right)$, $p_1 \in
  L^{r_1} \left( \mathbbm{R}^d \right)$, $p_2 \in L^{r_2} \left( \mathbbm{R}^d
  \right)$ with $1 \leqslant r_1, r_2 < \infty$. Suppose $P - p_1 - p_2$ is
  weakly harmonic, that is
  \begin{align*}
    \triangle \left( P - p_1 - p_2 \right) = 0
  \end{align*}
  in the sense of distribution, then there exists a constant $c$ such that
  \begin{align} \label{tmp_rev_e1}
    P - p_1 - p_2 = c \hspace{2em} a.e. \hspace{1em} x \in \mathbbm{R}^d .
  \end{align}
  If furthermore $P \left( x \right) \geqslant 0$ a.e., then we also have $c
  \geqslant 0$.
\end{lemma}

\begin{proof}
  This is fairly standard but we include the details here for the sake of
  completeness. By Weyl's lemma (see e.g. \cite{Jost2007}), the function $P
  - p_1 - p_2$ coincides almost everywhere with a strongly harmonic function
  which we denote by $\phi \left( x \right)$. Take any $x_1 \in \mathbbm{R}^d$
  and $R > 0$, denote by $B_R=B_R(x_1)$ the ball centered at $x_1$ with radius $R$.
  By using the mean-value property of harmonic functions, we have
  \begin{equation}
    \left| \phi \left( 0 \right) - \phi \left( x_1 \right) \right| =
    \frac{1}{\left| B_R \right|}  \left| \int_{\left| y \right| < R} \phi
    \left( y \right) \mathd y - \int_{\left| y - x_1 \right| < R} \phi \left(
    y \right) \mathd y \right| \leqslant \frac{1}{\left| B_R \right|}
    \int_{\tilde \Omega_R} \left| \phi \left( y \right) \right| \mathd y
  \end{equation}
  where
  \begin{align*}
  \tilde \Omega_R :=\Bigl( \{y:\, |y|<R\} \setminus \{y:\,|y-x_1|<R \} \Bigr)
  \bigcup \Bigl( \{y:\, |y-x_1|<R \} \setminus \{y:\, |y|<R\} \Bigr).
  \end{align*}
   Now as $\phi = P - p_1 - p_2$
  a.e., we get
  \begin{align}
    \frac{1}{\left| B_R \right|}  \int_{\tilde \Omega_R} \left| \phi \left( y \right)
    \right| \mathd y & \le\;  \frac{1}{\left| B_R \right|}  \left[
    \int_{\tilde \Omega_R} \left| P \right| \mathd y + \int_{\tilde \Omega_R} \left| p_1
    \right| \mathd y + \int_{\tilde \Omega_R} \left| p_2 \right| \mathd y \right]
    \nonumber\\
    & = \; \frac{1}{\left| B_R \right|}  \left[ \left\| P
    \right\|_{L^{\infty}}  \left| \tilde \Omega_R \right| + \left\| p_1
    \right\|_{L^{r_1}}  \left| \tilde \Omega_R \right|^{\frac{r_1 - 1}{r_1}} +
    \left\| p_2 \right\|_{L^{r_2}}  \left| \tilde \Omega_R \right|^{\frac{r_2 -
    1}{r_2}} \right] . \label{tmp_rev_e20}
  \end{align}
  As $\left| \tilde \Omega_R \right| \leqslant CR^{d - 1}$ for some constant $C$
  depending on $x_1$ only, we see that the right hand side $\rightarrow 0$ as
  $R \nearrow \infty$. Consequently $\phi \left( 0 \right) = \phi \left(
  x_1 \right)$ and $\phi\equiv$ constant.

  Finally if $P$ is nonnegative, we need to show $c\ge 0$. Take a nonnegative $\psi
  \in C_c^{\infty} \left( \mathbbm{R}^d \right)$ with $\| \psi\|_{L^1(\mathbb R^d)} \ne 0$.
  Taking convolution with $\psi$ on both sides of \eqref{tmp_rev_e1}, we have
  \begin{align}
    c \left\| \psi \right\|_{L^1} &= c \ast \psi = \left( P \ast \psi \right)
    \left( x \right) - \left( p_1 \ast \psi \right) \left( x \right) - \left(
    p_2 \ast \psi \right) \left( x \right) \notag \\
    &\ge\;
     - \left( p_1 \ast \psi
    \right) \left( x \right) - \left( p_2 \ast \psi \right) \left( x \right),
  \end{align}
  where we have used the positivity of $P$ and $\psi$. Since  by assumption $p_1 \in L^{r_1}$ and
$p_2 \in L^{r_2}$ with $1\le r_1,r_2<\infty$, we have $p_1*\psi \to 0$, $p_2*\psi \to 0$ as $|x| \to \infty$.
Hence $c\ge 0$.
\end{proof}

In the proof of Theorem \ref{thm:all-d}, we need the following variant of
Lemma \ref{lem:p2}.

\begin{lemma}
  \label{lem:div-p2}Suppose $P \in L^{\infty} \left( \mathbbm{R}^d \right)$,
  $p_1 \in L^{r_1} \left( \mathbbm{R}^d \right)$, $p_2 \in L^{r_2} \left(
  \mathbbm{R}^d \right)$ with $1 \leqslant r_1, r_2 < \infty$. Assume
  $\nabla \cdot p_2 \in L_{\tmop{loc}}^1 \left( \mathbbm{R}^d \right)$. If $P
  - p_1 - \nabla \cdot p_2$ is weakly harmonic, that is
  \begin{equation}
    \triangle \left( P - p_1 - \nabla \cdot p_2 \right) = 0
    \label{eq:lem-div-p2-eq}
  \end{equation}
  in the sense of distributions, then there is a constant $c$ such that
  \begin{align*}
    P - p_1 - \nabla \cdot p_2 = c \hspace{2em} a.e.\; x \in \mathbbm{R}^d .
  \end{align*}
  If furthermore $P \left( x \right) \geqslant 0$ for
  a.e. $x \in \mathbbm{R}^d$, then we also have $c \geqslant 0$.
\end{lemma}

\begin{proof}
  Note that Weyl's lemma still applies here since $P - p_1 - \nabla \cdot p_2$
  is locally integrable. Thus we have $\phi \in C^{\infty}$ strongly harmonic
  such that
  \begin{align}
    \phi = P - p_1 - \nabla \cdot p_2 \quad a.e.\, x\in \mathbb R^d. \label{eq:lem-24}
  \end{align}
Take a nonnegative $\psi \in C_c^{\infty}(\mathbb R^d)$ with $\psi(x)=1$ for $|x|\le 1$
and $\| \psi \|_{L_x^1}=1$. Define the standard mollifier $\psi_{\epsilon}(x) = \epsilon^{-d} \psi(x/ \epsilon)$
for $\epsilon>0$. Also denote $f_{\epsilon} = \psi_{\epsilon}*f$ for any function $f$.
 Taking convolution with $\psi_{\epsilon}$
on both sides of \eqref{eq:lem-24}, we get
\begin{align*}
\phi*\psi_{\epsilon}= P_{\epsilon}-p_{1\epsilon}- \nabla \cdot p_{2\epsilon}.
\end{align*}
Since $\phi$ is harmonic, by using the mean value property, we have $\phi*\psi_{\epsilon}=\phi$
and hence
\begin{align*}
\phi=P_{\epsilon}-p_{1\epsilon}- \nabla \cdot p_{2\epsilon}.
\end{align*}
By the same argument as in Lemma \ref{lem:p2} (see \eqref{tmp_rev_e20}), we get
\begin{align}
P_{\epsilon}-p_{1\epsilon}- \nabla \cdot p_{2\epsilon}=c, \label{lem2_tmp_1}
\end{align}
where $c$ is a constant depending on $\phi$ (and hence independent of $\epsilon$).
Since by assumption $p_1 \in L^{r_1}$, $\nabla \cdot p_2 \in L^1_{loc}$, we have
\begin{align*}
&p_{1\epsilon} \to p_1, \qquad a.e.\; x\in \mathbb R^d ,\notag \\
 &\nabla \cdot p_{2\epsilon} =(\nabla \cdot p_2)_{\epsilon} \to
\nabla \cdot p_2, \quad a.e.\; x\in \mathbb R^d.
\end{align*}
Therefore
\begin{align*}
P_1 - p_1-\nabla \cdot p_2 =c,\quad a.e.\, x \in \mathbb R^d.
\end{align*}
Finally if $P\ge 0$, then $P_{\epsilon}\ge 0$. By \eqref{lem2_tmp_1}, this yields $c\ge 0$.
\end{proof}

\section{Proof of Theorem \ref{thm:main}}

\subsection{Upgrading the integrability of $P$ and $\rho$}
To prepare for later estimates, we first study the integrability of $P$ and
$\rho$. Taking divergence on both sides of  (\ref{eq:NS-momentum}), we get
\begin{equation}
  \triangle \left( P - p_1 - p_2 \right) = 0, \label{eq:div-of-momentum}
\end{equation}
where
\begin{equation}
  p_1 : = \left( - \triangle \right)^{- 1} \left[ \sum_{i, j = 1}^d \partial_i
  \partial_j \left( \rho v_i v_j \right) \right], \hspace{2em} p_2 : = \left(
  2 \mu + \lambda \right) \nabla \cdot v. \label{eq:def-p1-p2}
\end{equation}
By assumption we have $p_2 \in L^2 \left( \mathbbm{R}^d \right)$.
For $p_1$, using Sobolev embedding $\dot{H}^1 \longhookrightarrow L^{\frac{2
d}{d - 2}}$ when $d \geqslant 4$ together with the assumptions
(\ref{eq:cond-d=4}--\ref{eq:cond-d=2}) we have
\begin{align}
  p_1 & \in \left\{ \begin{array}{ll}
    L^{\frac{d}{d - 2}} \left( \mathbbm{R}^d \right) & d \geqslant 4\\
    L^{9 / 4} \left( \mathbbm{R}^d \right) & d = 3\\
    L^3 \left( \mathbbm{R}^d \right) & d = 2
  \end{array} \right. .  \label{eq:p1-integrability}
\end{align}
Since $P \assign A \rho^{\gamma} \geqslant 0$, by Lemma \ref{lem:p2} we
get
\begin{align} \label{eq:P1_defa}
  A \rho^{\gamma} = P = c + p_1 + p_2
\end{align}
for some constant $c \geqslant 0$. Now consider the function
\begin{align}
  P_1 &:= \rho^{\gamma - 1} - \left( \frac{c}{A} \right)^{\frac{\gamma -
  1}{\gamma}}  \label{eq:P1-def} \\
  &=\; \left( \frac {c + p_1+p_2} A \right)^{\frac{\gamma-1}{\gamma}}
  - \left( \frac c A \right)^{\frac{\gamma-1}{\gamma}}. \notag
\end{align}
This function will be needed later in the integration by parts argument (see
\eqref{e_rev_56}).

Consider first $c>0$.  By the mean value theorem, we have
\begin{equation}
  \left| P_1 \right| \lesssim_{c, A} \left\{ \begin{array}{ll}
    \left| p_1 + p_2 \right| & \text{\tmop{if}} \left| p_1 + p_2 \right| < c /
    10,\\
    \left| p_1 + p_2 \right|^{\frac{\gamma - 1}{\gamma}} & \text{\tmop{if}}
    \left| p_1 + p_2 \right| \geqslant c / 10.
  \end{array} \right.
\end{equation}

Thus for $c>0$ we have the pointwise estimate
\begin{equation}
  \left| P_1 \right| \lesssim_{c, A} \left| p_1 \right| + \left| p_2 \right| .
 \notag %
\end{equation}
For $c=0$, note that by \eqref{eq:P1_defa}, we also have
\begin{align*}
P_1 \rho = \rho^{\gamma} \lesssim_A |p_1|+|p_2|.
\end{align*}
Therefore for all $c\ge 0$, we have
\begin{align}\label{eq:P1-est}
|P_1 \rho| \lesssim_{c,A,\|\rho\|_{\infty}} |p_1|+|p_2|.
\end{align}
This important pointwise estimate will be used later.

\subsection{Choice of the test function}

Let $\phi \in C_c^{\infty} \left( \mathbbm{R} \right)$ be an even function
such that
\begin{equation}
  \phi \left( s \right) \left\{ \begin{array}{ll}
    = 1 & \left| s \right| \leqslant \frac{1}{2}\\
    \in \left[ 0, 1 \right] & \left| s \right| \in \left( \frac{1}{2}, 1
    \right)\\
    = 0 & \left| s \right| \geqslant 1
  \end{array} \right., \label{eq:choice-phi}
\end{equation}
Let $\psi \in C_c^{\infty} \left( \mathbbm{R}^d \right)$ be a radial function
such that
\begin{equation}
  \psi \left( \left|x\right| \right) \left\{ \begin{array}{ll}
    = 1 & \left| x \right| \leqslant \frac{1}{2}\\
    \in \left[ 0, 1 \right] & \left| x \right| \in \left( \frac{1}{2}, 1
    \right)\\
    = 0 & \left| x \right| \geqslant 1
  \end{array} \right. . \label{eq:choice-psi}
\end{equation}
Now we take our test function as
\begin{equation}
  w \left( x \right) = v \left( x \right) \phi \left( \frac{\left| v(x)
  \right|^2}{R} \right) \psi \left( \frac{x}{R_1} \right)
  \label{eq:test-function}
\end{equation}
where $R, R_1 > 0$ will be taken to infinity in an appropriate order later.
The "Lebesgue" cut-off $\phi(|v|^2/R)$ is used to chop off high values of $v$.

Note
that we have
\begin{equation}
  \left| w \left( x \right) \right| \leqslant R\; \text{ \tmop{for all}} x
  \in \mathbbm{R}^d, \hspace{2em} w \left( x \right) = 0 \text{
  \tmop{for all}} \left| x \right| > R_1,
\end{equation}
and
\begin{equation}
  w \left( x \right) = v \left( x \right) \text{\tmop{for} \tmop{all}} x \in
  \left\{ \tilde x :\; \left| v \left( \tilde x \right) \right| \leqslant \left( \frac{R}{2}
  \right)^{1 / 2}, \hspace{1em} \left| \tilde x \right| \leqslant \frac{R_1}{2}
  \right\} .
\end{equation}

\subsection{Proof of Theorem \ref{thm:main}}

Taking the inner product of with $w \left( x \right)$ on
both sides of (\ref{eq:NS-momentum}), we
obtain
\begin{align}
  0  =& \; - \frac{1}{2}  \int_{\mathbbm{R}^d} \left( \rho v \cdot \nabla
  \right) \left( \left| v \right|^2 \right) \phi \left( \frac{v^2}{R} \right)
  \psi \left( \frac{x}{R_1} \right) \mathd x \nonumber\\
  &  \; - \int_{\mathbbm{R}^d} \nabla P \cdot v \phi \left( \frac{v^2}{R}
  \right) \psi \left( \frac{x}{R_1} \right) \mathd x \nonumber\\
  &  \; + \mu \int_{\mathbbm{R}^d} \triangle v \cdot v \phi \left(
  \frac{v^2}{R} \right) \psi \left( \frac{x}{R_1} \right) \mathd x \nonumber\\
  &  \; + \left( \mu + \lambda \right)  \int_{\mathbbm{R}^d} \left( \nabla
  \nabla \cdot v \right) \cdot v \phi \left( \frac{v^2}{R} \right) \psi \left(
  \frac{x}{R_1} \right) \mathd x \nonumber\\
  =& :
  \; I_1 + I_2 + I_3 + I_4 .
\end{align}
In the following we will take $R_1 \nearrow \infty$ first and then $R \nearrow
\infty$, and show that the above becomes
\begin{equation}
  0 = \mu \int_{\mathbbm{R}^d} \left| \nabla v \right|^2 \mathd x + \left( \mu
  + \lambda \right)  \int_{\mathbbm{R}^d} \left( \nabla \cdot v \right)^2
  \mathd x.
\end{equation}
This will imply $v =$constant. Combined with (\ref{eq:NS-momentum})
and the assumptions (\ref{eq:cond-d=4}--\ref{eq:cond-d=2}) we conclude that $v
= 0$ and $\nabla P = 0$. The latter then yields $\rho =$constant and ends the
proof.

We estimate $I_1 - I_4$ one by one. From now on we will adopt the following
notations to simplify the presentation of the proofs.
\begin{equation}
  \Omega_R \assign \text{supp} \Bigl( \phi \left( \frac{\left| v \right|^2}{R}
  \right)  \Bigr) \subseteq \left\{ x : \left| v \right|^2 \leqslant R \right\},
  \hspace{2em} 
  \label{eq:def-omega-r-r1}
\end{equation}
\begin{equation}\label{eq:def-c-r}
  C_R \assign \text{supp}\Bigl( \phi' \left( \frac{\left| v \right|^2}{R} \right) \Bigr)
  \subseteq \left\{ x : \frac{R}{2} \leqslant \left| v \right|^2 \leqslant R
  \right\}
  \end{equation}
  \begin{equation}
  C_{R_1}^{\prime} \assign \text{supp}\Bigl( \left( \nabla \psi
  \right) \left( \frac{x}{R_1} \right) \Bigr) \subseteq \left\{ x : \frac{R_1}{2}
  \leqslant \left| x \right| \leqslant R_1 \right\} . \label{eq:def-c-r-r1}
\end{equation}
It is easy to see that
\begin{equation}
  \left\| v \right\|_{L^{\infty} \left( \Omega_R \right)}, \hspace{1em}
  \left\| v \right\|_{L^{\infty} \left( C_R \right)} \leqslant R^{1 / 2} .
\end{equation}

\texttt{Estimate of $I_1$.} Let $\tilde{\phi} \left( z \right) \assign
\int_0^z \phi \left( s / R \right) \mathd s$. Clearly $\left| \tilde{\phi}
\left( z \right) \right| \leqslant \left| z \right|$.

We notice that
\begin{align}
  I_1 & =\;  - \frac{1}{2}  \int_{\mathbbm{R}^d} \left( \rho v \cdot \nabla
  \right) \left( \left| v \right|^2 \right) \phi \left( \frac{\left| v
  \right|^2}{R} \right) \psi \left( \frac{x}{R_1} \right) \mathd x \nonumber\\
  & = \; - \frac{1}{2}  \int_{\Omega_R} \rho v \cdot \nabla \tilde{\phi}
  \left( \left| v \right|^2 \right) \psi \left( \frac{x}{R_1} \right) \mathd x
  \nonumber\\
  & = \; \frac{1}{2}  \int_{\Omega_R} \left[ \rho v \tilde{\phi} \left( \left|
  v \right|^2 \right)  \right] \cdot \nabla \left[ \psi \left( \frac{x}{R_1}
  \right) \right] \mathd x \nonumber\\
  & = \;\frac{1}{2}  \int_{\Omega_R \cap C_{R_1}^{\prime}} \left[ \rho v \tilde{\phi}
  \left( \left| v \right|^2 \right)  \right] \cdot \left[ \frac{1}{R_1}
  \left( \nabla \psi \right) \left( \frac{x}{R_1} \right) \right] \mathd x.
  \text{}
\end{align}
Now using H\"older's inequality and the fact that $0 \leqslant \tilde{\phi}
\left( \left| v \right|^2 \right) \leqslant \left| v \right|^2$, we have
\begin{equation}
  \left| I_1 \right| \leqslant \frac{1}{2}  \left\| \rho \right\|_{L^{\infty}
  \left( \mathbbm{R}^d \right)}  \left\| v \right\|_{L^{\frac{3 d}{d - 1}}
  \left( \Omega_R \cap C_{R_1}^{\prime} \right)}^3  \left\| \frac{1}{R_1}  \left(
  \nabla \psi \right) \left( \frac{x}{R_1} \right) \right\|_{L^d \left(
  \mathbbm{R}^d \right)} .
\end{equation}
Since
\begin{equation}
  \left\| \frac{1}{R_1}  \left( \nabla \psi \right) \left( \frac{x}{R_1}
  \right) \right\|_{L^d \left( \mathbbm{R}^d \right)} = \left\| \nabla \psi
  \right\|_{L^d \left( \mathbbm{R}^d \right)}
\end{equation}
is independent of $R, R_1$, all we need to show is for any fixed $R$, $\left\|
v \right\|_{L^{\frac{3 d}{d - 1}} \left( \Omega_R \cap C_{R_1}^{\prime} \right)}
\longrightarrow 0$ as $R_1 \nearrow \infty$. It is clear that this is true if
\begin{equation}
  v \in L^{\frac{3 d}{d - 1}} \left( \Omega_R \right)
\end{equation}
over $\Omega_R$ as defined in (\ref{eq:def-omega-r-r1}).
\begin{itemize}
  \item $d \geqslant 4$. In this case we use Sobolev embedding $\dot{H}^1
  \left( \mathbbm{R}^d \right) \longhookrightarrow L^{\frac{2 d}{d - 2}}
  \left( \mathbbm{R}^d \right)$ and standard interpolation. Since $\frac{2
  d}{d - 2} \leqslant \frac{3 d}{d - 1}$ and $v \in L^{\infty} \left( \Omega_R
  \right)$, we have
  \begin{equation}
    v \in L^{\frac{2 d}{d - 2}} \left( \Omega_R \right) \cap L^{\infty} \left(
    \Omega_R \right) \Longrightarrow v \in L^{\frac{3 d}{d - 1}} \left(
    \Omega_R \right) .
  \end{equation}
  \item $d = 2, 3$. In these cases we have $\frac{3 d}{d - 1} = \frac{9}{2},
  6$ respectively. So $v \in L^{\frac{3 d}{d - 1}} \left( \Omega_R \right)$
  follows immediately from the assumptions (\ref{eq:cond-d=3},
  \ref{eq:cond-d=2}).
\end{itemize}
\texttt{Estimate of $I_2$}.

Observe
\begin{align} \label{e_rev_56}
  \nabla \left( \rho^{\gamma} \right) = \frac{\gamma}{\gamma - 1} \rho \nabla
  \left( \rho^{\gamma - 1} \right) = \frac{\gamma}{\gamma - 1} \rho \nabla P_1,
\end{align}
where $P_1$ was defined in \eqref{eq:P1-def}.

We have
\begin{align}
  I_2 & = \; - A \frac{\gamma}{\gamma - 1}  \int_{\mathbbm{R}^d} \rho \nabla
  P_1 \cdot v \phi \left( \frac{\left| v \right|^2}{R} \right) \psi \left(
  \frac{x}{R_1} \right) \mathd x \nonumber\\
  & = \; \frac{A \gamma}{\gamma - 1}  \int_{\mathbbm{R}^d} P_1  \left( \rho v
  \right) \cdot \left( \nabla \psi \left( \frac{x}{R_1} \right) \right)
  \frac{1}{R_1} \phi \left( \frac{\left| v \right|^2}{R} \right) \mathd x
  \nonumber\\
  & \;  + \frac{A \gamma}{\gamma - 1}  \int_{\mathbbm{R}^d} P_1 \rho
  \frac{2}{R} \phi' \left( \frac{\left| v \right|^2}{R} \right) \psi \left(
  \frac{x}{R_1} \right)  \sum_{j, k = 1}^d v_j v_k \partial_j v_k \mathd x
  \nonumber\\
  & = :  I_{2 a} + I_{2 b} .
\end{align}
For $I_{2 a}$ by using \eqref{eq:P1-est}, we have
\begin{align}
  \left|I_{2 a}\right| & = \left| \frac{A \gamma}{\gamma - 1}  \int_{C_{R_1}^{\prime}} P_1  \left( \rho v
  \right) \cdot \left( \nabla \psi \left( \frac{x}{R_1} \right) \right)
  \frac{1}{R_1} \phi \left( \frac{\left| v \right|^2}{R} \right) \mathd x\right|
  \nonumber\\
  & \lesssim_{A,\gamma}   \int_{C_{R_1}^{\prime}} \left| P_1 \rho \right|  \left| v \right| \phi \left(
  \frac{\left| v \right|^2}{R} \right)  \left[ \left( \nabla \psi \left(
  \frac{x}{R_1} \right) \right)  \frac{1}{R_1} \right] \mathd x \nonumber\\
  & \lesssim_{A,\gamma}   \left\| P_1 \rho  v \phi \left( \frac{\left| v \right|^2}{R} \right)
  \right\|_{L^{\frac{d}{d - 1}} \left( C_{R_1}^{\prime} \right)}  \left\| \left( \nabla
  \psi \left( \frac{x}{R_1} \right) \right)  \frac{1}{R_1} \right\|_{L^d}
  \nonumber\\
  & \lesssim_{A,\gamma,\|\rho\|_\infty}  \left\| p_1 v \right\|_{L^{\frac{d}{d - 1}} \left( \Omega_R
  \cap C_{R_1}^{\prime} \right)} + \left\| p_2 v \right\|_{L^{\frac{d}{d - 1}} \left(
  \Omega_R \cap C_{R_1}^{\prime} \right)} .
\end{align}
Here we have used the fact that the $\psi$ term is a constant depending
only on $\psi$. Now it suffices to show
\begin{equation}
  p_1 v \nocomma, \hspace{1em} p_2 v \in L^{\frac{d}{d - 1}} \left( \Omega_R
  \right) \nosymbol .
\end{equation}
For the first term we recall
\begin{align}
  p_1 & \in  \left\{ \begin{array}{ll}
    L^{\frac{d}{d - 2}} \left( \Omega_R \right) & d \geqslant 4\\
    L^{9 / 4} \left( \Omega_R \right) & d = 3\\
    L^3 \left( \Omega_R \right) & d = 2
  \end{array} \right.,  \label{eq:201207061044}
\end{align}
and, as $R$ is fixed now,
\begin{equation}
  v \in L^{\infty} \left( \Omega_R \right) \cap \left\{ \begin{array}{ll}
    L^{\frac{2 d}{d - 2}} \left( \Omega_R \right) & d \geqslant 4\\
    L^{9 / 2} \left( \Omega_R \right) & d = 3\\
    L^6 \left( \Omega_R \right) & d = 2
  \end{array} \right. \Longrightarrow v \in L^q \left( \Omega_R \right)
  \hspace{2em} \text{\tmop{for} \tmop{any}}  \left\{ \begin{array}{ll}
    q \geqslant \frac{2 d}{d - 2} & d \geqslant 4\\
    q \geqslant \frac{9}{2} & d = 3\\
    q \geqslant 6 & d = 2
  \end{array} \right. \label{eq:201207061045}
\end{equation}
by interpolation. With these integrability properties we now proceed as
follows.
\begin{itemize}
  \item $d \geqslant 4$. In this case we have $d \geqslant \frac{2 d}{d - 2}$.
  Thus following (\ref{eq:201207061044}), (\ref{eq:201207061045})
  \begin{equation}
    p_1 \in L^{\frac{d}{d - 2}} \left( \Omega_R \right), \hspace{1em} v \in
    L^d \left( \Omega_R \right) \Longrightarrow p_1 v \in L^{\frac{d}{d - 1}}
    \left( \Omega_R \right) ;
  \end{equation}
  Furthermore
  \begin{equation}
    p_2 \in L^2 \left( \Omega_R \right), \hspace{1em} v \in L^{\frac{2 d}{d -
    2}} \left( \Omega_R \right) \Longrightarrow p_2 v \in L^{\frac{d}{d - 1}}
    \left( \Omega_R \right) .
  \end{equation}
  \item $d = 3$. In this case $\frac{d}{d - 1} = \frac{3}{2}$. We have
  \begin{equation}
    p_1 \in L^{9 / 4} \left( \Omega_R \right), \hspace{1em} v \in L^{9 / 2}
    \left( \Omega_R \right) \Longrightarrow p_1 v \in L^{\frac{d}{d - 1}}
    \left( \Omega_R \right) ;
  \end{equation}
  \begin{equation}
    p_2 \in L^2 \left( \Omega_R \right), \hspace{1em} v \in L^6 \left(
    \Omega_R \right) \Longrightarrow p_2 v \in L^{\frac{d}{d - 1}} \left(
    \Omega_R \right) .
  \end{equation}
  \item $d = 2$. In this case $\frac{d}{d - 1} = 2$. We have
  \begin{equation}
    p_1 \in L^3 \left( \Omega_R \right), \hspace{1em} v \in L^6 \left(
    \Omega_R \right) \Longrightarrow p_1 v \in L^{\frac{d}{d - 1}} \left(
    \Omega_R \right) ;
  \end{equation}
  \begin{equation}
    p_2 \in L^2 \left( \Omega_R \right), \hspace{1em} v \in L^{\infty} \left(
    \Omega_R \right) \Longrightarrow p_1 v \in L^{\frac{d}{d - 1}} \left(
    \Omega_R \right) .
  \end{equation}
\end{itemize}

Next we estimate $I_{2b}$. By the definitions  of $C_R,\phi,\psi,P_1$ (see \eqref{eq:def-c-r},
\eqref{eq:choice-phi}, \eqref{eq:choice-psi}), \eqref{eq:P1-def}) and H\"older's inequality we have
\begin{align}
   \left| I_{2b} \right| & = \left| \frac{A \gamma}{\gamma - 1}  \int_{\mathbbm{R}^d} P_1 \rho
  \frac{2}{R} \phi' \left( \frac{\left| v \right|^2}{R} \right) \psi \left(
  \frac{x}{R_1} \right)  \sum_{j, k = 1}^d v_j v_k \partial_j v_k \mathd x \right| \notag\\
 & \lesssim_{A,\gamma} \int_{C_R} \left| P_1 \rho \right|  \frac{|v|^2}{R} |\nabla v| \mathd x
\notag \\
 & \lesssim_{A,\gamma} \int_{C_R} \left| \left(\rho^{\gamma - 1} - \left( \frac{c}{A} \right)^{\frac{\gamma -1}{\gamma}}\right)
 \rho \right| \left| \nabla v\right| \mathd x
 \notag\\
 &\lesssim_{c,A,\gamma,\|\rho\|_{\infty}} \left\| 1 \right\|_{L^2(C_R)} \left\| \nabla v\right\|_{L^2} \notag \\
 & = |C_R|^{\frac 12} \left\| \nabla v\right\|_{L^2}. \notag
\end{align}
Thus we only need to show $|C_R|$ goes to $0$ as $R\nearrow\infty$. Thanks to Chebyshev's inequality,
\begin{equation}
|C_R| \le |\left\{x: |v(x)|\ge R/2 \right\} \lesssim R^{-p} \|v\|_p^p,
\end{equation}
it suffices to check $v\in L^p(\mathbb{R}^d)$ for some $p>0$. Verification of this is quite straightforward:
\begin{itemize}
        \item $d\ge 4$. $v\in L^{\frac{2d}{d-2}}(\mathbb{R}^d)$ thanks to the Sobolev embedding $\dot{H}^1
  \left( \mathbbm{R}^d \right) \longhookrightarrow L^{\frac{2 d}{d - 2}}
  \left( \mathbbm{R}^d \right)$;
        \item $d=3$. $v\in L^{9/2}(\mathbb{R}^d)$ by assumptions of the theorem;
        \item $d=2$. $v\in L^6(\mathbb{R}^d)$ by assumptions of the theorem.
\end{itemize}
Therefore $|C_R|\rightarrow 0$ as $R\nearrow \infty$ in all cases. Consequently $I_{2b}\rightarrow 0$ as desired.

\texttt{Estimate of $I_3$.}

We have
\begin{align}
  I_3  =&\;  \mu \int_{\mathbbm{R}^d} \triangle v \cdot v \phi \left(
  \frac{v^2}{R} \right) \psi \left( \frac{x}{R_1} \right) \mathd x \nonumber\\
   =&\;  - \mu \int_{\mathbbm{R}^d}  \left| \nabla v \right|^2 \phi \left(
  \frac{v^2}{R} \right) \psi \left( \frac{x}{R_1} \right) \mathd x \nonumber\\
  & \;  - \mu \int_{\mathbbm{R}^d} \left( \nabla v \right) v \cdot \nabla
  \left[ \phi \left( \frac{v^2}{R} \right) \right] \psi \left( \frac{x}{R_1}
  \right) \mathd x \nonumber\\
  &\;   - \mu \int_{\mathbbm{R}^d} \left( \nabla v \right) v \cdot \nabla
  \left[ \psi \left( \frac{x}{R_1} \right) \right] \phi \left( \frac{v^2}{R}
  \right) \mathd x \nonumber\\
   =& :  I_{3 a} + I_{3 b} + I_{3 c} .
\end{align}
First consider $I_{3 a}$. As $\nabla v \in L^2 \left( \mathbbm{R}^d \right)$
it is clear from Lebesgue dominated convergence that

\begin{equation}
  I_{3a} \longrightarrow - \mu \int_{\mathbbm{R}^d} \left| \nabla v \right|^2
  \mathd x. \hspace{2em} \text{as } R_1 \nearrow \infty
  \text{ followed by } R \nearrow \infty .
\end{equation}

For $I_{3 b}$, we have by dominated convergence
\begin{align}
  \left| I_{3 b} \right| & \lesssim \; \int_{\mathbbm{R}^d} \left|
  \nabla v
  \right|^2  \frac{\left| v \right|^2}{R} \left| \phi'
  \left( \frac{\left| v\right|^2}{R} \right) \right| \mathd x \nonumber
  \\
  &\lesssim \; \int_{C_R} \left| \nabla v \right|^2 \mathd x
  \rightarrow 0
  \hspace{2em} \text{as } R_1 \nearrow \infty
  \text{ followed by } R \nearrow \infty.
\end{align}

For $I_{3 c}$, we have
\begin{align}
  \left| I_{3 c} \right| & \lesssim \; \int_{C_{R_1}^{\prime}} \left| \nabla v \right|
  \left| v \right| \phi \left( \frac{v^2}{R} \right)  \left| \frac{1}{R_1}
  \nabla \psi \right| \mathd x \nonumber\\
  & \lesssim \; \left\| \nabla v \right\|_{L^2 \left( C_{R_1}^{\prime} \right)}
  \left\| v \right\|_{L^{\frac{2 d}{d - 2}} \left( C_{R_1}^{\prime} \cap \Omega_R \right) }  \left\|
  \frac{1}{R_1} \nabla \psi \right\|_{L^d \left( \mathbbm{R}^d \right)}
  \rightarrow 0 \hspace{2em} \text{as } R_1 \nearrow \infty
\end{align}
since $v \in L^{\frac{2 d}{d - 2}} \left( \mathbbm{R}^d \right)$ by Sobolev
embedding when $d \geqslant 3$. When $d = 2$ note that as $R$ is fixed, $v \in
L^{\infty} \left( \Omega_{R} \right)$ so the above still holds.

\texttt{Estimate of $I_4$.}

This is similar to the estimate of $I_3$ and hence we omit the details.

Collecting the estimates, we have
\begin{equation}
  0 = I_1 + \cdots + I_4 \longrightarrow - \mu \int_{\mathbbm{R}^d} \left|
  \nabla v \right|^2 \mathd x - \left( \mu + \lambda \right)
  \int_{\mathbbm{R}^d} \left( \nabla \cdot v \right)^2 \mathd x
\end{equation}
which gives $\nabla v = 0$ and therefore $v =$constant. Together with the
assumptions (\ref{eq:cond-d=4}--\ref{eq:cond-d=2}) we have $v = 0$.
From the momentum equation (\ref{eq:NS-momentum})
and the relation (\ref{eq:NS-pressure}) we also obtain $\rho \equiv$ constant. This ends the proof of Theorem
\ref{thm:main}.

\section{Proofs of Theorems \ref{thm:all-d} and \ref{thm:div-v}}

\subsection{Proof of Theorem \ref{thm:all-d}}

Taking the dot product with $v$ on both sides of (\ref{eq:NS-momentum}), we get
\begin{equation}
  \rho v \cdot \nabla \left( \frac{\left| v \right|^2}{2} \right) = - A
  \frac{\gamma }{\gamma-1} \rho v \cdot \nabla \left( \rho^{\gamma - 1}
  \right) + \mu v \cdot \triangle v + \left( \mu + \lambda \right) v \cdot
  \nabla \left( \nabla \cdot v \right) . \label{eq:NS-momentum-multiply-v}
\end{equation}
Using $\nabla \cdot \left( \rho v \right) = 0$ and the identities
\begin{align}
  \triangle \left( \left| v \right|^2 \right) & = \; 2 v \cdot \triangle v + 2
  \left| \nabla v \right|^2, \\
  \nabla \cdot \left( v \nabla \cdot v \right) & = \; v \cdot \nabla \left(
  \nabla \cdot v \right) + \left| \nabla \cdot v \right|^2,
\end{align}
we can rewrite (\ref{eq:NS-momentum-multiply-v}) as
\begin{align}
  \frac{1}{2} \nabla \cdot \left( \rho v \left| v \right|^2 \right) & = \; -
  \frac{\gamma }{\gamma-1} A \nabla \cdot \left( \rho^{\gamma} v \right) +
  \mu \left( \frac{1}{2} \triangle \left( \left| v \right|^2 \right) - \left|
  \nabla v \right|^2 \right) \nonumber\\
  &  \; + \left( \mu + \lambda \right)  \left( \nabla \cdot \left( v \nabla
  \cdot v \right) - \left| \nabla \cdot v \right|^2 \right) \nonumber\\
  & = \; - \frac{\gamma }{\gamma-1} \nabla \cdot \left( Pv \right) + \mu
  \left( \frac{1}{2} \triangle \left( v^2 \right) - \left| \nabla v \right|^2
  \right) \nonumber\\
  &  \; + \left( \mu + \lambda \right) \left( \nabla \cdot \left( v \nabla
  \cdot v \right) - \left| \nabla \cdot v \right|^2 \right) .
  \label{eq:NS-momentum-multiply-v-1}
\end{align}
Now let $\psi$ be the cut-off function as defined in (\ref{eq:choice-psi}).
Multiplying (\ref{eq:NS-momentum-multiply-v-1}) by $\psi \left( \frac{x}{R_1}
\right)$ and integrating by parts, we obtain
\begin{align}
  0  = & \frac{1}{2}  \int_{\mathbbm{R}^d} \rho \left| v \right|^2 v \cdot
  \left( \nabla \psi \right) \left( \frac{x}{R_1} \right)  \frac{1}{R_1}
  \mathd x \nonumber\\
    &\; + \frac{\gamma }{\gamma-1}  \int_{\mathbbm{R}^d} Pv \cdot \left(
  \nabla \psi \right) \left( \frac{x}{R_1} \right)  \frac{1}{R_1} \mathd x
  \nonumber\\
  &  \; + \frac{\mu}{2}  \int_{\mathbbm{R}^d} v^2  \frac{1}{R_1^2}  \left(
  \triangle \psi \right) \left( \frac{x}{R_1} \right) \mathd x \nonumber\\
  &  \; - \mu \int_{\mathbbm{R}^d} \left| \nabla v \right|^2 \psi \left(
  \frac{x}{R_1} \right) \mathd x - \left( \mu + \lambda \right)
  \int_{\mathbbm{R}^d} \left| \nabla \cdot v \right|^2 \psi \left(
  \frac{x}{R_1} \right) \mathd x \nonumber\\
  &  \; - \left( \mu + \lambda \right)  \int_{\mathbbm{R}^d} \left( \nabla
  \cdot v \right) v \cdot \left( \nabla \psi \right) \left( \frac{x}{R_1}
  \right)  \frac{1}{R_1} \mathd x \nonumber\\
   = : & \; I_1 + \cdots + I_5 .
\end{align}
We shall estimate $I_1,\cdots,I_5$ one by one.

For $I_1$ we have
\begin{align}
  \left| I_1 \right|  \leqslant &\; \frac{1}{2}  \int_{C_{R_1}^{\prime}} \rho \left| v
  \right|^3  \left| \left( \nabla \psi \right) \left( \frac{x}{R_1} \right)
  \frac{1}{R_1} \right| \mathd x \nonumber\\
   \leqslant &\; \frac{1}{2}  \left\| \rho \right\|_{L^{\infty}}  \left\| v
  \right\|^3_{L^{\frac{3 d}{ d - 1}} \left( C_{R_1}^{\prime} \right)}  \left\|
  \left( \nabla \psi \right) \left( \frac{x}{R_1} \right)  \frac{1}{R_1}
  \right\|_{L^d} .
\end{align}
Here $C_{R_1}^{\prime}$ is the support of $\nabla \psi \left( \frac{x}{R_1} \right)$,
as defined in (\ref{eq:def-c-r-r1}). By assumption we have $v \in L^{\frac{3 d}
{d-1}} \left( \mathbbm{R}^d \right)$, thus $\left\| v
\right\|_{L^{\frac{3d}{d-1}} \left( C_{R_1}^{\prime} \right)}
\longrightarrow 0$ as $R_1 \nearrow \infty$. So $I_1 \rightarrow 0$ as $R_1
\nearrow \infty$.

For $I_2$ we have
\begin{align}
  \left| I_2 \right| \leqslant & \; \frac{\gamma }{\gamma-1}  \left\| P
  \right\|_{L^{\infty}}  \left\| v \right\|_{L^{d / \left( d - 1 \right)}
  \left( C_{R_1}^{\prime} \right)}  \left\| \left( \nabla \psi \right) \left(
  \frac{x}{R_1} \right)  \frac{1}{R_1} \right\|_{L^d} .
\end{align}
By assumption $v \in L^{d / \left( d - 1 \right)} \left( \mathbbm{R}^d
\right)$ so this term goes to $0$ too.

For $I_3$ we notice that the assumption $v \in L^{\frac{3 d}{d - 1}} \left(
\mathbbm{R}^d \right) \cap L^{\frac{d}{d - 1}} \left( \mathbbm{R}^d \right)$
leads to
\begin{equation}
  v \in L^2 \left( \mathbbm{R}^d \right)
\end{equation}
through interpolation (When $d = 2$ we already have $\frac{d}{d - 1} = 2$).
Thus
\begin{align}
  \left| I_3 \right|  \leqslant & \; \frac{\mu}{2}  \left\| v^2 \right\|_{L^1
  \left( \mathbbm{R}^d \right)}  \left\| \frac{1}{R_1^2}  \left( \triangle
  \psi \right) \left( \frac{x}{R_1} \right) \right\|_{L^{\infty} \left(
  \mathbbm{R}^d \right)} \nonumber\\
   \lesssim & \; \left\| v \right\|_{L^2 \left( \mathbbm{R}^d \right)}^2 R_1^{-
  2} \longrightarrow 0, \hspace{2em} \text{as } R_1 \nearrow \infty .
\end{align}
For $I_4$ we notice that
\begin{equation}
  \left| \nabla v \right|^2 \psi \left( \frac{x}{R_1} \right) \nocomma,
  \hspace{1em} \left| \nabla \cdot v \right|^2 \psi \left( \frac{x}{R_1}
  \right) \geqslant 0
\end{equation}
and furthermore at every $x$ are increasing with respect to $R_1$.\footnote{We can
take a subsequence in $R_1$ if necessary so that the function $\psi(x/R_1)$ is increasing.}
 An application
of Lebesgue's monotone convergence theorem then gives
\begin{equation}
  I_4 \longrightarrow - \mu \int_{\mathbbm{R}^d} \left| \nabla v \right|^2
  \mathd x - \left( \mu + \lambda \right) \int_{\mathbbm{R}^d} \left| \nabla
  \cdot v \right|^2 \mathd x
\end{equation}
as $R_1 \nearrow \infty$.

For $I_5$ we first recall (\ref{eq:div-of-momentum}-\ref{eq:def-p1-p2}),
which are obtained from taking divergence of (\ref{eq:NS-momentum}):
\begin{equation}
  \triangle \left( P - p_1 - \nabla \cdot p_2 \right) = 0, \hspace{2em} p_1 :
  = \left( - \triangle \right)^{- 1} \left[ \sum_{i, j = 1}^d \partial_i
  \partial_j \left( \rho v_i v_j \right) \right], \hspace{1em} p_2 : = \left(
  2 \mu + \lambda \right) v.
\end{equation}
The assumptions of Theorem \ref{thm:all-d} give
\begin{equation}
  P \in L^{\infty} \left( \mathbbm{R}^d \right), \hspace{2em} p_1 \in L^{ \frac {3 d }
  {2 ( d - 1 )}} \left( \mathbbm{R}^d \right), \hspace{2em} v \in
  L^{\frac {3 d }{ d - 1 }} \left( \mathbbm{R}^d \right) .
\end{equation}
Applying Lemma \ref{lem:div-p2} we obtain
\begin{equation}
  \nabla \cdot v = \frac{1}{2 \mu + \lambda}  \left[ P - p_1 - c \right] \in
  L^{\infty} \left( \mathbbm{R}^d \right) + L^{ \frac {3 d} {2 ( d - 1)}}
  \left( \mathbbm{R}^d \right) .
\end{equation}
From this it is easy to see that $I_5$ enjoys the same estimates as $I_1 +
I_2$. More precisely, we have
\begin{align}
  \left| I_5 \right|  \lesssim &  \; \int_{C_{R_1}^{\prime}} \left| P - c \right|  \left|
  v \right|  \left| \left( \nabla \psi \right) \left( \frac{x}{R_1} \right)
  \frac{1}{R_1} \right| \mathd x \nonumber\\
  &  \;+ \int_{C_{R_1}^{\prime}} \left| p_1 \right|  \left| v \right|  \left| \left(
  \nabla \psi \right) \left( \frac{x}{R_1} \right)  \frac{1}{R_1} \right|
  \mathd x \nonumber\\
   \leqslant & \; \left\| P - c \right\|_{L^{\infty}}  \left\| v \right\|_{L^{d
  / \left( d - 1 \right)} \left( C_{R_1}^{\prime} \right)}  \left\| \left( \nabla \psi
  \right) \left( \frac{x}{R_1} \right)  \frac{1}{R_1} \right\|_{L^d}
  \nonumber\\
  &  \; + \left\| p_1 \right\|_{L^{3 d / 2 \left( d - 1 \right)} \left(
  C_{R_1}^{\prime} \right)}  \left\| v \right\|_{L^{3 d / \left( d - 1 \right)} \left(
  C_{R_1}^{\prime} \right)}  \left\| \left( \nabla \psi \right) \left( \frac{x}{R_1}
  \right)  \frac{1}{R_1} \right\|_{L^d} \nonumber\\
    \longrightarrow & \; 0, \hspace{2em} \text{\tmop{as}} R_1 \nearrow \infty .
\end{align}
Collecting the estimates of $I_1--I_5$, we conclude that as $R_1 \nearrow
\infty$, we have
\begin{equation}
  0 = I_1 + \cdots + I_5 \longrightarrow - \mu \int_{\mathbbm{R}^d} \left|
  \nabla v \right|^2 \mathd x - \left( \mu + \lambda \right)
  \int_{\mathbbm{R}^d} \left| \nabla \cdot v \right|^2 \mathd x
\end{equation}
which gives $\nabla v = 0$ and the conclusions of Theorem \ref{thm:all-d}
follow.

\subsection{Proof of Theorem \ref{thm:div-v}}

Application of Lemma \ref{lem:p2} gives
\begin{equation}
  A \rho^{\gamma} = P = p_1 + \left( 2 \mu + \lambda \right) \nabla \cdot v +
  c \hspace{2em} a.e.\  x \in \mathbbm{R}^d .
\end{equation}
Defining $P_1$ as in (\ref{eq:P1-def}) and following the estimate
(\ref{eq:P1-est}), we have
\begin{equation}
  \left| P_1 \rho \right| \lesssim \left| p_1 \right| + \left| \nabla \cdot v
  \right| . \label{eq:P1-est-div-v}
\end{equation}
Recall that
\begin{equation}
  p_1 : = \left( - \triangle \right)^{- 1} \left[ \sum_{i, j = 1}^d \partial_i
  \partial_j \left( \rho v_i v_j \right) \right] \in L^{\frac{3 d}{ 2 ( d - 1
  )}} \left( \mathbbm{R}^d \right)
\end{equation}
by the assumption on $v$ and the boundedness of Riesz operators.

We rewrite (\ref{eq:NS-momentum-multiply-v}) as
\begin{align}
 & \frac{1}{2} \nabla \cdot \left( \rho v \left| v \right|^2 \right) \notag\\
   =&\; -
  \frac{A \gamma }{\gamma-1} \nabla \cdot \left( P_1 \rho v
  \right) + \mu \left( \frac{1}{2} \triangle \left( \left| v \right|^2 \right)
  - \left| \nabla v \right|^2 \right) + \left( \mu + \lambda \right)  \left(
  \nabla \cdot \left( v \nabla \cdot v \right) - \left| \nabla \cdot v
  \right|^2 \right) . \label{eq:NS-momentum-multiply-v-variant}
\end{align}
and pick the cut-off function $\psi$ as before with a parameter $R_1 > 0$
which will tend to $\infty$ later.

Multiplying (\ref{eq:NS-momentum-multiply-v-variant}) by $\psi \left(
\frac{x}{R_1} \right)$ and integrating by parts, we reach
\begin{align}
  0 = &\; \frac{1}{2}  \int_{\mathbbm{R}^d} \rho \left| v \right|^2 v \cdot
  \left( \nabla \psi \right) \left( \frac{x}{R_1} \right)  \frac{1}{R_1}
  \mathd x \nonumber\\
  &  \;+ A \frac{\gamma }{\gamma-1}  \int_{\mathbbm{R}^d} P_1 \rho v \cdot
  \left( \nabla \psi \right) \left( \frac{x}{R_1} \right)  \frac{1}{R_1}
  \mathd x \nonumber\\
  &  \; + \frac{\mu}{2}  \int_{\mathbbm{R}^d} \left| v \right|^2
  \frac{1}{R_1^2}  \left( \triangle \psi \right) \left( \frac{x}{R_1} \right)
  \mathd x \nonumber\\
  &  \; - \mu \int_{\mathbbm{R}^d} \left| \nabla v \right|^2 \psi \left(
  \frac{x}{R_1} \right) \mathd x - \left( \mu + \lambda \right)
  \int_{\mathbbm{R}^d} \left| \nabla \cdot v \right|^2 \psi \left(
  \frac{x}{R_1} \right) \mathd x \nonumber\\
  &  \; - \left( \mu + \lambda \right)  \int_{\mathbbm{R}^d} \left( \nabla
  \cdot v \right) v \cdot \left( \nabla \psi \right) \left( \frac{x}{R_1}
  \right)  \frac{1}{R_1} \mathd x \nonumber\\
   = : & \; I_1 + \cdots I_5 .
\end{align}
Now we estimate them one by one.

For $I_1$, we have
\begin{equation}
  \left| I_1 \right| \lesssim \left\| \rho \right\|_{L^{\infty} \left(
  \mathbbm{R}^d \right)}  \left\| v \right\|_{L^{\frac{3 d}{ d - 1}}
  \left( C_{R_1}^{\prime} \right)}^3  \left\| \left( \nabla \psi \right) \left(
  \frac{x}{R_1} \right)  \frac{1}{R_1} \right\|_{L^d \left( \mathbbm{R}^d
  \right)} \longrightarrow 0 \text{ as } R_1 \nearrow \infty
\end{equation}
because $v \in L^{3 d / \left( d - 1 \right)} \left( \mathbbm{R}^d \right)$ by
assumption.

For $I_2$ we apply the estimate (\ref{eq:P1-est-div-v}) to obtain
\begin{align}
  \left| I_2 \right| \lesssim &\; \int_{\mathbbm{R}^d} \left| p_1 \right|
  \left| v \right|  \left| \left( \nabla \psi \right) \left( \frac{x}{R_1}
  \right)  \frac{1}{R_1} \right| \mathd x + \int_{\mathbbm{R}^d} \left| \nabla
  \cdot v \right|  \left| v \right|  \left| \left( \nabla \psi \right) \left(
  \frac{x}{R_1} \right)  \frac{1}{R_1} \right| \mathd x \nonumber\\
   = : &\; I_{2 a} + I_{2 b} .
\end{align}
For $I_{2 a}$ we have
\begin{equation}
  I_{2 a} \lesssim \left\| p_1 \right\|_{L^{ \frac{3 d}  {2 ( d - 1 )}}
  \left( C_{R_1}^{\prime} \right)}  \left\| v \right\|_{L^{\frac{3 d} { d - 1}}
  \left( C_{R_1}^{\prime} \right)}  \left\| \left( \nabla \psi \right) \left(
  \frac{x}{R_1} \right)  \frac{1}{R_1} \right\|_{L^d \left( \mathbbm{R}^d
  \right)} \longrightarrow 0 \text{ as } R_1 \nearrow \infty .
\end{equation}
For $I_{2 b}$ we consider two cases:
\begin{itemize}
  \item $d \geqslant 4$. We estimate it as
  \begin{equation}
    I_{2 b} \lesssim \left\| \nabla \cdot v \right\|_{L^2 \left( C_{R_1}^{\prime}
    \right)}  \left\| v \right\|_{L^{\frac{2 d } {d - 2}} \left(
    C_{R_1}^{\prime} \right)}  \left\| \left( \nabla \psi \right) \left( \frac{x}{R_1}
    \right)  \frac{1}{R_1} \right\|_{L^d \left( \mathbbm{R}^d \right)}
    \longrightarrow 0 \text{ as } R_1 \nearrow \infty .
  \end{equation}
  \item $d \le 3$. In this case we have
  \begin{equation}
    I_{2 b} \lesssim \left\| \nabla \cdot v \right\|_{L^2 \left( C_{R_1}^{\prime}
    \right)}  \left\| v \right\|_{L^{\frac{3d}{d-1}} \left(
    C_{R_1}^{\prime} \right)}  \left\| \left( \nabla \psi \right) \left( \frac{x}{R_1}
    \right)  \frac{1}{R_1} \right\|_{L^{6 d / \left( d + 2 \right)} \left(
    \mathbbm{R}^d \right)} \longrightarrow 0 \text{ as } R_1 \nearrow
    \infty
  \end{equation}
  since when $d \le 3$ we have $\frac{6 d}{d + 2} > d$ and therefore
  \begin{equation}
    \left\| \left( \nabla \psi \right) \left( \frac{x}{R_1} \right)
    \frac{1}{R_1} \right\|_{L^{6 d / \left( d + 2 \right)} \left(
    \mathbbm{R}^d \right)} \longrightarrow 0 \text{\tmop{as}} R_1 \nearrow
    \infty \nosymbol .
  \end{equation}
\end{itemize}
For $I_3$ we estimate it as follows:
\begin{itemize}
  \item $d \geqslant 5$.
  \begin{align}
    \left| I_3 \right| \lesssim &\; \left\| \left| v \right|^2 \right\|_{L^{d
    / \left( d - 2 \right)} \left( C_{R_1}^{\prime} \right)}  \left\| \frac{1}{R_1^2}
    \left( \triangle \psi \right) \left( \frac{x}{R_1} \right) \right\|_{L^{d
    / 2} \left( \mathbbm{R}^d \right)} \nonumber\\
    = & \;\left\| v \right\|_{L^{2 d / \left( d - 2 \right)} \left( C_{R_1}^{\prime}
    \right)}^2  \left\| \frac{1}{R_1^2}  \left( \triangle \psi \right) \left(
    \frac{x}{R_1} \right) \right\|_{L^{d / 2} \left( \mathbbm{R}^d \right)}
    \longrightarrow 0 \text{ as } R_1 \nearrow \infty .
  \end{align}
  \item $d \leqslant 4$.
  \begin{align}
    \left| I_3 \right| \lesssim &\; \left\| \left| v \right|^2 \right\|_{L^{3
    d / 2 \left( d - 1 \right)} \left( \mathbbm{R}^d \right)}  \left\|
    \frac{1}{R_1^2}  \left( \triangle \psi \right) \left( \frac{x}{R_1}
    \right) \right\|_{L^{3 d / \left( d + 2 \right)} \left( \mathbbm{R}^d
    \right)} \nonumber\\
     = &\; \left\| v \right\|_{L^{3 d / \left( d - 1 \right)} \left( C_{R_1}^{\prime}
    \right)}^2  \left\| \frac{1}{R_1^2}  \left( \triangle \psi \right) \left(
    \frac{x}{R_1} \right) \right\|_{L^{3 d / \left( d + 2 \right)} \left(
    \mathbbm{R}^d \right)} .
  \end{align}
  When $d \leqslant 4$ we have $\frac{3 d}{d + 2} \geqslant \frac{d}{2}$ so
  $I_3 \longrightarrow 0$.
\end{itemize}
For $I_4$ the same argument for $I_4$ in the proof of Theorem \ref{thm:all-d}
works, and we have
\begin{equation}
  I_4 \longrightarrow - \mu \int_{\mathbbm{R}^d} \left| \nabla v \right|^2
  \mathd x - \left( \mu + \lambda \right) \int_{\mathbbm{R}^d} \left| \nabla
  \cdot v \right|^2 \mathd x \text{ as } R_1 \nearrow \infty .
\end{equation}
Finally, since $I_5$ is essentially the same as $I_{2 b}$, we have
\begin{equation}
  I_5 \longrightarrow 0 \text{ as } R_1 \nearrow \infty .
\end{equation}
Collecting all the estimates, we have
\begin{equation}
  0 = I_1 + \cdots + I_5 \longrightarrow - \mu \int_{\mathbbm{R}^d} \left|
  \nabla v \right|^2 \mathd x - \left( \mu + \lambda \right)
  \int_{\mathbbm{R}^d} \left| \nabla \cdot v \right|^2 \mathd x.
\end{equation}
Consequently $\nabla v = 0$ and the conclusions of Theorem \ref{thm:div-v}
follow.

{\bf Acknowledgement.} This work was initiated at University of British Columbia when the second author visited there. The research of D. Li is partly supported by NSERC Discovery grant and a start-up grant from University of British Columbia. The research of X. Yu is partly supported by NSERC Discovery grant and a start-up grant from University of Alberta.

\end{document}